\documentclass{rmmcart}

\usepackage{graphicx}

\title{The topology of tile invariants}

\author{Michael P. Hitchman}

\address{Linfield College}

\email{mhitchm@linfield.edu}

\date{June 5, 2012}

\volno{45}
\volissno{2}

\pubyear{2015}
\doipublisherid{10.1216}
\copyrightholder{Rocky Mountain Mathematics Consortium}
\publisherpubid{RMJ-2015-45-2-539}

\theoremstyle{plain}
\newtheorem{theorem}{Theorem}[section]
\newtheorem{lemma}[theorem]{Lemma}

\begin{document}
\rmj 
\setcounter{page}{539}
\maketitle
\begin{abstract}
In this note we use techniques in the topology of 2-complexes to recast some tools that have arisen in the study of planar tiling questions.  With spherical pictures we show that the tile counting group associated to a set $T$ of tiles and a set of regions tileable by $T$ is isomorphic to a quotient of the second homology group of a 2-complex built from $T$.  In this topological setting we derive some well-known tile invariants, one of which we apply to the solution of a tiling question involving modified rectangles.
\end{abstract}
%%%%%%%%%%%%% INTRODUCTION %%%%%%%%%%%%%%%%%%%%
\section{Introduction}

Two papers published in 1990, one by Conway and Lagarias \cite{conway} and one by Thurston \cite{thurston}, introduced topological and combinatorial group theoretic ideas to the solution of certain tiling questions.  About ten years later a tile counting group was developed in a series of papers by Pak and others (see \cite{korn, moore, muchnik, pak, sheffield}), in an effort to study the relations among the tiles that must persist in any tiling of a given region.
In this note we show that for a large class of tile sets and tileable regions, the tile counting group is isomorphic to a quotient of the second cellular homology group of a 2-complex associated to the tile set (Theorem \ref{th:1}).  We then use techniques in low-dimensional topology to revisit the derivation of tile invariants found in \cite{conway, pak} and subsequently used to solve tiling problems.  Finally, we use tile invariants to solve (in Theorem \ref{th:4}) the following tiling question:

For what values of $a$ and $b$ can the modified rectangle $M(a,b)$ (obtained by removing the upper left and lower right squares from an $a \times b$ rectangle in the integer lattice) be tiled by translations of the tiles in the set $S$ pictured in Figure \ref{fig:Sintro}?

\begin{figure}[h]
\centering
\centerline{\includegraphics[width=1.3in]{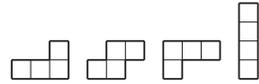}}
\caption{The tile set $S$.}\label{fig:Sintro}
\end{figure}

The general scene in this note is as follows.  We have a set of tiles $T = \{t_1, t_2, \ldots, t_n\}$ in the integer lattice.  We assume each tile $t_i$ consists of a finite union of squares 
in the lattice and that the resulting region is homeomorphic to a disk.  Ribbon tiles, which have been widely studied, fit into this category (see \cite{conway, korn, moore, muchnik,pak, reid, sheffield, thurston}).
A {\em ribbon tile} of area $n$ consists of $n$
squares laid out in a path, such that if we start with the bottom left-most square, each step either goes up or to the right.
 If $T_n$ denotes the set of ribbon tiles having area $n$, we note that $T_n$ will have $2^{n-1}$ members.  Figure \ref{fig:T3} shows the set $T_3$ of ribbon tile trominoes, and Figure \ref{fig:T4} shows the set $T_4$ of ribbon tile tetrominoes.

\begin{figure}[h]
\centering
\centerline{\includegraphics{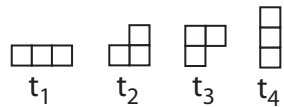}}
\caption{The set $T_3$ of ribbon tile trominoes.}\label{fig:T3}

\end{figure}

\begin{figure}[h]
\centering
\centerline{\includegraphics{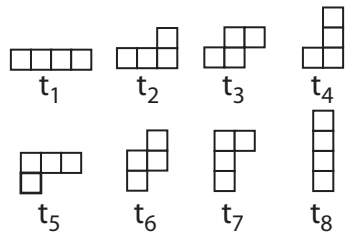}}
\caption{The set $T_4$ of ribbon tile tetrominoes.}\label{fig:T4}
\end{figure}

Suppose $\Gamma$ is a region that can be tiled by a tile set $T$.  By this we mean that we may cover the region without gaps or overlaps with translations of copies of the tiles in $T$.  We cannot rotate the tiles in $T$.   We restrict our attention to those regions $\Gamma$ that are homeomorphic to a disk (such regions are also called simply connected polyominoes in tiling circles) because we wish to take the view that tilings represent maps from a disk to a particular 2-complex.

If $\alpha$ represents a particular tiling by $T$ of a region $\Gamma$, we let $b_i(\alpha)$ equal the number of copies of tile $t_i$ that appear in the tiling.  There are certain relations among the $b_{i}(\alpha)$ that hold for all tilings of a given region, and any such relation is called a tile invariant in \cite{pak}. One simple relation is an area invariant: If all tiles in $T$ have the same area (number of squares), then for any tiling $\alpha$ of $\Gamma$, $\sum_{i=1}^n b_i(\alpha)$ is constant.  In \cite{conway}, Conway and Lagarias prove that for any tiling $\alpha$ of a given simply connected region $\Gamma$ by the set $T_3$ of ribbon tile trominoes given in Figure \ref{fig:T3}, $b_3(\alpha)-b_2(\alpha)$ is constant. We revisit this result below (Theorem \ref{th:T3}).

The {\bf tile counting group} $G(T,\mathcal{R})$ associated to a tile set $T$ and a given set of regions $\mathcal{R}$ is an abelian group whose elements correspond to tile invariants.  This group is defined in \cite{pak} as follows.  Begin with all formal integer linear combinations of the tiles in $T$, and then divide by $I$, the linear span of all relations of the form $$\sum_{i=1}^n b_i(\alpha)t_i - \sum_{i=1}^nb_i(\beta)t_i,$$ for all tilings $\alpha$ and $\beta$ of regions $\Gamma$ in $\mathcal{R}$.  That is,   $$G(T,\mathcal{R}) = \bigoplus_{i=1}^n \mathbb{Z} t_i \bigg/ I.$$
In a series of papers (\cite{moore, mechanic, pak}), Pak and others determine the tile counting group for the set $T_n$ (for $n \geq 2$) over all simply connected tileable regions.  In \cite{korn}, Korn determines the tile counting group for several other interesting tile sets over various families of tileable regions.
We introduce the picture group in the next section to approach the tile counting group topologically.

%%%%%%%%%%%%%%%%  THE PICTURE GROUP

\section{The Picture Group}

Suppose $\mathcal{P} = \langle {\bf x} : {\bf r}\rangle $ presents the group $G$.  We let $F$ denote the free group generated by the set of letters ${\bf x}$, and we let $R$ denote the normal subgroup of $F$ generated by the words in ${\bf r}$ (called relators) so that $G = F/R$.  The standard 2-complex $X_{\mathcal{P}}$ associated to $\mathcal{P}$ consists of a single 0-cell $c^0$, a 1-cell $c^1_x$ for each generator $x \in {\bf x}$, and a 2-cell $c^2_r$ for reach relator $r \in {\bf r}$.  The boundary of each 1-cell must be attached to $c^0$ so that the 1-skeleton of $X_{\mathcal P}$, $X^1_{\mathcal{P}}$, is a bouquet of circles.  The 2-cell $c^2_r$ is attached to $X^1_{\mathcal{P}}$ along its boundary according to the word spelled by the relator.

For instance, the standard 2-complex associated to $\mathcal{P} = \langle x, y ~|~ [x,y]\rangle$ of the group $\mathbb{Z} \times \mathbb{Z}$ is the torus (here $[x,y]$ denotes the word $xyx^{-1}y^{-1}$).  The 1-skeleton of the space consists of two circles joined at the 0-cell.  The space has a single 2-cell attached to the 1-skeleton according to the word $xyx^{-1}y^{-1}$.  We may model this process by identifying opposite edges of a rectangle in the usual way to obtain a torus.

The standard 2-complex $X_{\mathcal{P}}$ associated to a presentation is not so easy to visualize in general, but we may represent elements of $\pi_2(X_{\mathcal{P}})$ with tangible objects called spherical pictures. A good introduction to pictures may be found in \cite{thebook} but we present the main ideas in this section.

\begin{figure}[h]
\centering
\centerline{\includegraphics[width=3in]{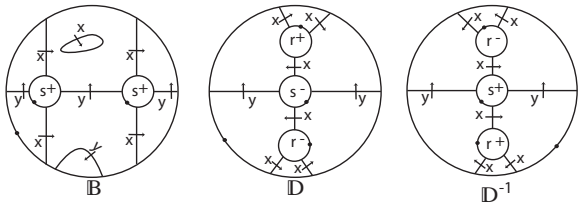}}
\caption{Some pictures over $\mathcal{P} = \langle x,y ~|~r=x^3, s = [y,x]\rangle$.}\label{fig:pic}
\end{figure}

A {\em picture} over a presentation $\mathcal{P}$ has an ambient disk with basepoint, and a finite number $\Delta_1, \Delta_2, \ldots, \Delta_k$ of based disks in the interior of the ambient disk.  In addition to a basepoint, each $\Delta_i$ is labelled with a relator and an exponent of $\pm 1$.  A picture also includes a finite number of arcs.  The two ends of an arc are either attached to disk boundaries or joined to each other (creating a loop).  No two arcs may intersect, and no arc can intersect the interior of any of the $\Delta_i$ disks.  Furthermore, each arc gets a label from the generator set and a transverse orientation.  The disk and arc labelling must be consistent in this sense:  traversing the boundary of the disk $\Delta_i$ in the clockwise direction from its basepoint, the word one spells from the arc labels it encounters equals the disk label (using the convention that the exponent of the generator associated to an arc is $+1$ or $-1$ depending on whether the path around $\Delta_i$ is in the direction of the transverse arrow associated to the arc, or against this direction).

A picture in which no arc intersects the boundary of the ambient disk is called a {\em spherical picture}.  A picture $\mathbb{B}$ over $\mathcal{P}$ has a boundary word $\partial \mathbb{B}$ in $F$ associated with it, read by the generator labels encountered by traversing the boundary of the ambient disk once in the clockwise direction from the ambient basepoint.  By a lemma of van Kampen \cite{vk} this boundary word is trivial in the group presented by $\mathcal{P}$.

\begin{lemma} (van Kampen)  Suppose $\mathcal{P}$ presents the group $F/R$.  A word $w \in F$ is trivial in the group $F/R$ if and only if there exists a picture $\mathbb{B}$ over the presentation $\mathcal{P}$ whose boundary word $\partial\mathbb{B}$ is identically equal to $w$.
\end{lemma}

Figure \ref{fig:pic} shows three pictures over  $\mathcal{P} = \langle x,y ~|~x^3, [y,x]\rangle$ of the group $\mathbb{Z} \times \mathbb{Z}_3$, where we let $r$ denote the relator $x^3$, and $s$ denotes the relator $[y,x]$.  Note that $\partial \mathbb{B}$ and $\partial \mathbb{D}$ are both equal to $yx^2y^{-1}x^{-2}$ in $F$.

We define an equivalence relation on pictures over $\mathcal{P}$ by $\mathbb{B}_1 \sim \mathbb{B}_2$ if and only if $\mathbb{B}_1$ can be transformed to $\mathbb{B}_2$ (up to planar isotopy) by a finite number of moves of three possible types, called FLOAT, BRIDGE, and FOLD moves.  A FLOAT move is the insertion or deletion of a floating arc, an arc that divides the ambient disk into two components, one of which contains the global basepoint and all the disks and remaining arcs of the picture.  The picture $\mathbb{B}$ in Figure \ref{fig:pic} has two floating arcs, one labeled by $y$ that hits the boundary of the ambient disk, and one loop labeled by $x$.   A BRIDGE move is a local move with respect to two arcs in a picture sharing a generator label as indicated in Figure \ref{fig:picmove}.  A FOLD move is the insertion or deletion of a folding pair of disks (see Figure \ref{fig:picmove}).  A folding pair of disks consists of two disks, labelled by the same relator but with opposite signs, such that each arc of one disk connects to the other, and such that the disk basepoints are in the same region of the picture.

\begin{figure}[h]
\centering
\centerline{\includegraphics[width=2.2in]{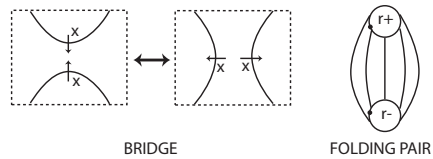}}
\caption{BRIDGE and FOLD moves}\label{fig:picmove}
\end{figure}

Note that when two disks of a picture have arcs that meet the boundary of the ambient disk successively as suggested in Figure \ref{fig:bridgefloat}, and if those arcs are labeled by the same generator but with opposite orientations, we may replace these two arcs by a single arc between the disks by performing a BRIDGE move followed by a FLOAT.

\begin{figure}[h]
\centering
\centerline{\includegraphics[width=3.5in]{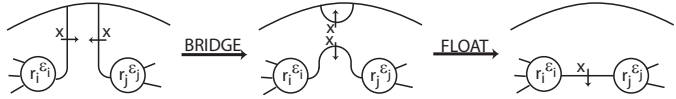}}
\caption{Joining two disks by a BRIDGE, FLOAT sequence}\label{fig:bridgefloat}
\end{figure}

We may multiply two pictures by concatenation so that $\partial(\mathbb{B}\cdot \mathbb{D}) = \partial \mathbb{B} \cdot \partial\mathbb{D}$ (see Figure \ref{fig:picprod}). Given $w \in F$, the picture $w \cdot \mathbb{B}$  is obtained by surrounding the ambient basepoint of $\mathbb{B}$ by a series of arcs whose labels spell $w$  (see Figure \ref{fig:picprod}).  Note that $\partial(w\cdot \mathbb{B}) = w \partial \mathbb{B} w^{-1}$.  The inverse $\mathbb{D}^{-1}$ of a picture is obtained by planar reflection of $\mathbb{D}$ and a change of sign of the disks of $\mathbb{D}$ (Figure \ref{fig:pic} shows a picture $\mathbb{D}$ and its inverse $\mathbb{D}^{-1}$).

\begin{figure}[h]
\centering
\centerline{\includegraphics[width=2.4in]{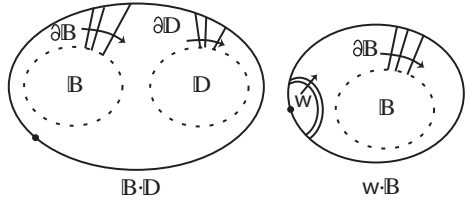}}
\caption{The product of two pictures, the $F$ action on $\Sigma$.}\label{fig:picprod}
\end{figure}

To get a better feel for the arithmetic of pictures, consider again the pictures from Figure \ref{fig:pic}.  The product of $\mathbb{B}$ (with its floating arcs removed) and $\mathbb{D}^{-1}$ is given in Figure \ref{fig:pics2}.  By a series of BRIDGE and FLOAT moves along the lines of Figure \ref{fig:bridgefloat} this picture is equivalent to the spherical picture at right in Figure \ref{fig:pics2}. That the result is a spherical picture is no accident:  if $\mathbb{B}$ and $\mathbb{D}$ are based pictures over $\mathcal{P}$ such that $\partial\mathbb{B}$ and $\partial\mathbb{D}$ determine the same word in the free group $F$, then $\mathbb{B}\cdot \mathbb{D}^{-1}$ is equivalent to a spherical picture over $\mathcal{P}$.

\begin{figure}[h]
\centering
\centerline{\includegraphics[width=3in]{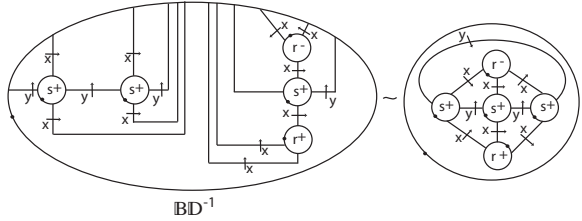}}
\caption{The picture $\mathbb{B}\cdot \mathbb{D}^{-1}$ is spherical.}\label{fig:pics2}
\end{figure}

Let $[\mathbb{B}]$ denote the equivalence class of $\mathbb{B}$, and define $\Sigma = \{ [\mathbb{B}] ~|$ $\mathbb{B} {\rm ~ a ~ picture~over~} \mathcal{P}\}$.  It is not difficult to check that $\Sigma$ forms a group under multiplication given by $[\mathbb{B}] \cdot [\mathbb{D}]=[\mathbb{B} \cdot \mathbb{D}]$.  The identity element of $\Sigma$ is $[\mathbb{O}]$ where $\mathbb{O}$ is the empty picture (consisting of just an ambient disk). One checks that $[\mathbb{B}]^{-1}=[\mathbb{B}^{-1}]$ and that $w\cdot [\mathbb{B}] = [w\cdot \mathbb{B}]$ defines an $F$-action on $\Sigma$.  We call $\Sigma$ the group of pictures over $\mathcal{P}$.

The set $P ~=~ \{ [\mathbb{P}] ~~|~~ \mathbb{P} ~{\rm a~spherical~picture~over~}\mathcal{P}\}$ determines a normal subgroup of $\Sigma$ which we call the group of spherical pictures over $\mathcal{P}$.

Pictures over $\mathcal{P}$ give combinatorial representations of continuous maps from the unit disk (or the 2-sphere if the picture is spherical) into $X_{\mathcal{P}}$. One can show that two pictures represent homotopy equivalent maps if and only they determine the same element in the picture group.  In fact, the group of pictures is isomorphic to the relative homotopy group $\pi_2(X_{\mathcal{P}}, X_{\mathcal{P}}^1)$, and the group of spherical pictures is isomorphic to $\pi_2(X_{\mathcal{ P}})$.   We refer the reader to \cite{thebook} for details.  Thus, elements of $\pi_2(X_{\mathcal{P}})$ may be represented by spherical pictures, and elements of $\pi_2(X_{\mathcal{P}}, X_{\mathcal{P}}^1)$ may be represented by (disk) pictures.  %We assume homotopy groups are based at the single 0-cell of the 2-complex.

The group of spherical pictures $\pi_2(X_{\mathcal{P}})$ has a left $\mathbb{Z} G$-module structure.  One can check that if $w \in R$ and $\mathbb{P}$ is spherical then $w \cdot \mathbb{P} \sim \mathbb{P}$.  That is, $R$ acts trivially on $\pi_2(X_{\mathcal{P}})$ so we have a well-defined $G$-action on $\pi_2(X_{\mathcal{P}})$ that extends to give $\pi_2(X_{\mathcal{P}})$ a left $\mathbb{Z} G$-module structure.  When we discuss generators of $\pi_2(X_{\mathcal{P}})$ in what follows we mean module generators.

\section{The tile complex}

We return now to our tile set $T = \{t_1, t_2, \ldots,$ $t_n\}$ in the integer lattice and the task of building a 2-complex associated to the tile set. The 2-complex will be modeled on a group presentation which consists of two generators and $n$ relators, one for each tile in $T$.  In particular, for each $t_i$ chose a basepoint on its boundary, and assign a boundary word in the free group $F = F(x,y)$ from a clockwise path around its boundary as follows: record the letter $y$ (resp. $y^{-1}$) if a vertical edge is traversed to the north (resp. south); and record the letter $x$ (resp. $x^{-1}$) if a horizontal edge is traversed to the east (resp. west).  Let $r_i ~(i = 1, 2, \ldots, n)$ denote the $n$ boundary words of the $n$ tiles.  The group presentation $$\mathcal{P} = \langle x,y~|~r_1, r_2, \ldots, r_n \rangle$$  presents the {\bf tile boundary group} $G_T = F/R$, where $R$ is the normal subgroup of $F$ generated by the $r_i$.  The word $r_i$ depends on the choice of tile basepoint, but words associated to distinct basepoints will be cyclic permutations of each other, and the group determined by the resulting presentation is independent of this choice of basepoint.
To any presentation of the tile boundary group we may associate the standard 2-complex, and we are interested in the homotopy and homology groups associated to this complex.  Again, these groups are independent of our choice of tile basepoints.

By placing tile basepoints at the southwest corner of each tile, the tile boundary group for the set of ribbon tile trominoes $T_3$ (Figure \ref{fig:T3}) has the presentation:
$$\mathcal{P}_{3} = \langle x,y ~|~ [y,x^3],~ (yx)^2y^{-2}x^{-2},~ y^2x^2(y^{-1}x^{-1})^2,~[y^3,x]\rangle.$$
Using southwest corners again for the tiles of $T_4$ (Figure \ref{fig:T4}), its tile boundary group has the 8 relator presentation
$$\mathcal{P}_{4} = \langle x,y ~|~ [y,x^4],~ yx^2yxy^{-2}x^{-3},~ yxyx^2y^{-1}x^{-1}y^{-1}x^{-2}, \ldots,~[y^4,x]\rangle.$$

Thurston considered the tile boundary group in \cite{thurston} for dominoes and lozenges, and Conway and Lagarias worked with the group $[F,F]/R$ in \cite{conway} (which they called the tile homotopy group), where $[F,F]$ denotes the subgroup of $F$ generated by all commutators $[w,v]$.

Reid studied the Conway/Lagarias tile homotopy group with great success in his paper \cite{reid}.    In a sense, the tile boundary group $F/R$ we consider in this note (which Reid calls the tile path group of the tile set) is bigger than we need to analyze tiling questions because simply connected regions in the integer lattice have boundary words that live in $[F,F]$.  For this reason Conway and Lagarias and Reid considered $[F,F]/R$.  However, it turns out that the tile complex associated to a presentation of $F/R$ is well-suited to interpret Pak's tile counting group, so we focus on this group in what follows.

Suppose we have some presentation $\mathcal{P}$ of the tile boundary group $G_T$ based on a choice of tile basepoints. Any tiling $\alpha$ by $T$ of a region $\Gamma$ that is topologically a disk has a natural representation as a based picture over $\mathcal{P}$, and hence determines an element in the picture group $\pi_2(X_{\mathcal{P}},X_{\mathcal{P}}^1)$. In fact, one may think of the original tiling as a van Kampen diagram over $\mathcal{P}$, and the corresponding picture as a planar dual.  In particular, think of the topological boundary of $\Gamma$ as the ambient disk, and each tile in the tiling as a disk in the picture.  Tile $t_i$ gets replaced with a disk labelled $r_i^+$ and two disks share an arc between them if their corresponding tiles share an edge.  Any edge of a tile that intersects the boundary of $\Gamma$ gives rise to an arc from its disk to the ambient disk.  The arcs are labeled according to the boundary letter of the associated edge of the tile. If $\alpha$ employs $k$ tiles to tile $\Gamma$, the resulting disk picture $\mathbb{B}_\alpha$ has $k$ disks, all with positive exponents.

\begin{figure}[h]
\centering
\includegraphics[width=2in]{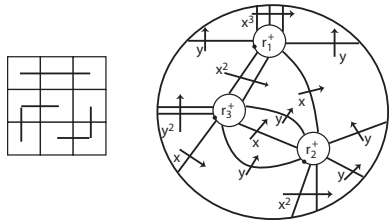}
\caption{Converting a tiling of a region by the tromino set $T_3$ into a picture over $\mathcal{P}_3$.}\label{fig:tilepic}
\end{figure}

Two tilings $\alpha$ and $\beta$ of a given region $\Gamma$ determine a spherical picture $\mathbb{B}_\alpha\cdot \mathbb{B}_\beta^{-1}$ over $\mathcal{P}$, and the equivalence class $[\mathbb{B}_\alpha\cdot \mathbb{B}_\beta^{-1}]$ is an element of the group of spherical pictures, $\pi_2(X_{\mathcal{P}})$.

We note that an arbitrary picture over $\mathcal{P}$, even one with all positive disks, need not correspond to a tiling of a region in the lattice.  Also, tilings of other simply connected regions that are not homeomorphic to a disk (e.g., a region that is the one point union of two such disk regions) may be encoded in the picture group.
Passing to a more general object to investigate tilings can be helpful in so much as we can bring general tools of topology to bear on the picture group, which is not an easy task, in general.

Let $\mathcal{R}$ consist of all regions homeomorphic to a disk, and suppose $X_{\mathcal{P}}$ is the tile complex associated to some presentation of the tile boundary group $G_T$, where  $T = \{t_1, \ldots, t_n\}$ is the tile set.   Let $J$ denote the submodule of $\pi_2(X_{\mathcal{P}})$ generated by all spherical pictures of the form $\mathbb{B}_\alpha\cdot \mathbb{B}_\beta^{-1}$ for all tilings $\alpha$ and $\beta$ for all regions $\Gamma$ in $\mathcal{R}$.  We let $H_2(X_{\mathcal{P}})$ denote the second cellular homology group of $X_{\mathcal{P}}$.

\begin{theorem}\label{th:1} The tile counting group $G(T,\mathcal{R})$ is isomorphic to the quotient $H_2(X_{\mathcal{ P}})/h(J)$, where $h:\pi_2(X_{\mathcal{P}}) \to H_2(X_{\mathcal{P}})$ is the Hurewicz homomorphism.\end{theorem}

\begin{proof}
The cellular chain groups of $X_{\mathcal{P}}$ are $$C_2(X_{\mathcal{P}}) = \bigoplus_{i=1}^n \mathbb{Z} c^2_{r_i},$$ 

$$C_1(X_{\mathcal {P}}) =  \mathbb{Z} c^1_x\oplus \mathbb{Z} c^1_y,$$ 

$$C_0(X_{\mathcal{P}})=\mathbb{Z} c^0.$$

\noindent Since $X_{\mathcal{P}}$ has a single 0-cell, the first boundary map $\partial_1: C_1(X_{\mathcal{P}}) \to C_0(X_{\mathcal{P}})$ is trivial.  Since each tile in $T$ determines a closed path in the integer lattice, its boundary word, whatever the basepoint, is an element of $[F,F]$.  That is, the exponent sum of each generator in the word $r_i$ is zero.  Thus, the second boundary map is also trivial, and it follows that
$$H_2(X_{\mathcal{P}}) = \bigoplus_{i = 1}^n \mathbb{Z} c^2_{r_i}$$

and 

$$H_1(X_{\mathcal{P}}) = \mathbb{Z} \oplus \mathbb{Z}.$$

We may view the Hurewicz homomorphism through the lens of spherical pictures.  Following the notation in \cite{thebook}, given a picture $\mathbb{P}$ over $\mathcal{P}$ we record the (net) number of occurrences of disks labeled by relator $r$ as follows:
$${\rm exp}_r(\mathbb{P}) = (\# {\rm~disks ~in~} \mathbb{P} {\rm ~labelled~} r^+) - (\# {\rm ~disks~ in~} \mathbb{P} {\rm ~labelled~} r^-).$$
Note that if $\mathbb{P}_1 \sim \mathbb{P}_2$ then ${\rm exp}_r(\mathbb{P}_1)= {\rm exp}_r(\mathbb{P}_2)$ since BRIDGE, FLOAT and FOLD moves preserve the (net) number of disks labeled by $r$ in a picture.
Then the Hurewicz homomorphism $$h:\pi_2(X_{\mathcal{P}}) \to H_2(X_{\mathcal{P}})$$ is defined by $$h([\mathbb{P}]) = \sum_{i=1}^n {\rm exp}_r(\mathbb{P}) c^2_r.$$

\noindent Now consider a spherical picture over $\mathcal{P}$ of the form $\mathbb{B}_\alpha \cdot \mathbb{B}_\beta^{-1}$, coming from tilings $\alpha$ and $\beta$ of some region $\Gamma$ in $\mathcal{R}$.  Since each relator in $\mathbb{B}_\alpha$ has positive exponent, and each relator in $\mathbb{B}_\beta^{-1}$ has negative exponent, it follows that $${\rm exp}_{r_i}(\mathbb{B}_\alpha \cdot \mathbb{B}_\beta^{-1}) = b_i(\alpha) - b_i(\beta).$$

\noindent That is, the homological image of this picture under the Hurewicz map is $$h([\mathbb{B}_\alpha \cdot \mathbb{B}_\beta^{-1}]) = \sum_{i =1}^n (b_i(\alpha) - b_i(\beta))c^2_{r_i}.$$  

\noindent Thus, each element in $I$ is naturally identified with an element of $h(J)$, and it follows that 

\hskip1.2in{$G(T,\mathcal{R}) = H_2(X_{\mathcal{P}})/h(J).$}\end{proof}

\smallskip

We view Theorem \ref{th:1} primarily as a theoretical result meant to establish a topological interpretation of the tile counting group,  not as a result that will be generally useful for calculating the tile counting group.  Still, it is useful to explore the connection between these perspectives a bit further by looking at some special cases and revisiting a few well-known examples.

If $\pi_2(X_{\mathcal{P}})$ is generated by spherical pictures that come from the difference of two tilings of a region $\Gamma$, that is, if $\pi_2(X_{\mathcal{P}}) = J$, then the tile counting group is isomorphic to $H_2(G_T)$, the second homology group of the tile boundary group.  This follows from a theorem of Hopf (see, for instance \cite[p. 41]{brown}) which gives us the following exact sequence of groups $$\pi_2(X_{\mathcal{P}}) \xrightarrow{h} H_2(X_{\mathcal{P}}) \to H_2(G_T) \to 0.$$

In fact, all we need to ensure that $G(T,\mathcal{R})$ is isomorphic to $H_2(G_T)$ is that $J$ and $\pi_2(X_{\mathcal{P}})$ have the same image under the Hurewicz map.  An interesting topological question for a tile set is to decide whether this is the case.

In the event that the Hurewicz map $h: \pi_2(X_{\mathcal{P}}) \to H_2(X_{\mathcal{P}})$ is trivial, in which case the space $X_{\mathcal{P}}$ is called {\em Cockcroft},  we note that the tile counting group is isomorphic to $H_2(X_{\mathcal{P}})$, a free abelian group of rank equal to the number of tiles in the tile set.

1. {\em Dominoes}.  Consider the set $T_2 = \{t_1, t_2\}$ of dominoes at left in Figure \ref{fig:dominoes}.  Using southwest corners as basepoints, the domino tile counting group has presentation $$\mathcal{P}_2 = \langle x, y ~|~ [y,x^2], [y^2,x] \rangle.$$  Thurston argued in \cite{thurston} using the universal cover of $X_{\mathcal{P}_2}$ that $\pi_2(X_{\mathcal{P}_2})$ is generated by the spherical picture appearing in Figure \ref{fig:dominoes}.  This spherical picture results from the difference of two tilings of a 2 $\times$ 2 square, one by two horizontal tiles and one by two vertical tiles, as indicated at right in Figure \ref{fig:dominoes}.  Thus, $\pi_2(X_{\mathcal{P}_2}) = J$ in the case of dominoes, and the tile counting group is $H_2(G_{T_2})=\mathbb{Z} \times \mathbb{Z}_2$.

\begin{figure}[h]
\centering
\includegraphics[width=3in]{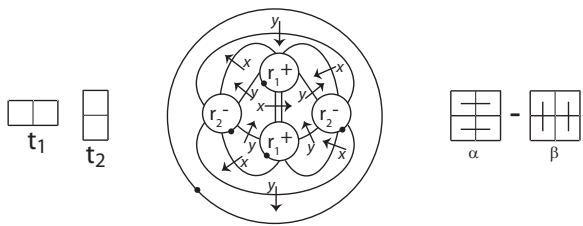}
\caption{$\pi_2(X_{\mathcal{P}_2})$ is generated by a single picture which lives in $J$.}\label{fig:dominoes}
\end{figure}

As Thurston observed, this analysis of $\pi_2(X_{\mathcal{P}_2})$ gives a local replacement property for tilings involving dominoes.  Given two tilings $\alpha$ and $\beta$ of any region $\Gamma$ homeomorphic to a disk, then $\mathbb{B}_\alpha \cdot \mathbb{B}_\beta^{-1}$ determines an element of $\pi_2(X_{\mathcal{P}_2})$ and hence must be expressible in terms of the picture in Figure \ref{fig:dominoes}. This translates to the notion that any domino tiling of a region $\Gamma$ may be obtained from any other tiling of $\Gamma$ by a finite sequence of ``2-flips'' that exchange a horizontally tiled 2 $\times$ 2 square portion of $\Gamma$ with a vertically tiled portion, or vice versa.

2. {\em Ribbon Tiles}.  In \cite{sheffield}, Sheffield proved the deep result that any two tilings of a simply connected region by $T_n$ can be made to coincide by a number of ``2-flips'' chosen from a finite list of local moves.  For instance, any tiling of a given region by the set $T_3$ can be changed into any other tiling of that region by a number of exchanges of two tiles from among the 6 indicated in Figure \ref{fig:T32flips}.  If a tile set has this local move property with respect to the collection of all tileable regions homeomorphic to a disk, then the local moves generate $J$.   For instance, the 6 spherical pictures determined by the 2-flips in Figure \ref{fig:T32flips} generate $J$ in the case of $T_3$.  See \cite{korn} for other examples of tile sets having a local move property, as well as examples of tile sets that do not have a local move property. We note that Sheffield also proves in \cite{sheffield} the existence of a linear time algorithm for deciding whether a given simply connected region is tileable by $T_n$.

\begin{figure}[h]
\centering
\includegraphics[width=2.5in]{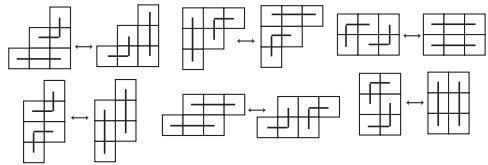}
\caption{These 6 2-flips generate $J$ for the tile set $T_3$.}\label{fig:T32flips}
\end{figure}

3. {\em Lozenges.}  We may construct tile boundary groups for tile sets in the triangular lattice from the free group on three generators, $F = F(x,y,z)$.  For instance, the tile boundary group for the tile set $L=\{l_1, l_2, l_3\}$ consisting of the three lozenges pictured in Figure \ref{fig:lozenge}  has presentation  $$\mathcal{L} = \langle x,y,z~|~r_1 = [x,y], r_2 = [y,z], r_3 = [z,x]\rangle.$$

Theorem \ref{th:1} need not apply to tile sets in the triangular lattice because boundary words of such tiles need not live in $[F,F]$ (consider a triangular tile, for instance).  However, the theorem does apply to tile sets whose boundary words do live in $[F,F]$ such as the lozenges.

The tile boundary group for lozenges is isomorphic to $\mathbb{Z}^3$. These lozenges were studied in \cite{thurston} as the dominoes were, and $\pi_2(X_{\mathcal{L}})$ may be understood by considering the universal cover of $X_{\mathcal{L}}$, $\widetilde{X_{\mathcal{L}}}$, since $\pi_2(X_{\mathcal{L}})$ is isomorphic to $H_2(\widetilde{X_{\mathcal{L}}})$.  Here $\widetilde{X_{\mathcal{L}}}$ is the 2-skeleton of a cubical tessellation of $\mathbb{R}^3$.  By an examination of the 2-dimensional cellular cycles of the universal cover, one checks that the picture in Figure \ref{fig:lozenge} generates $\pi_2(X_{\mathcal{L}})$.  This picture is equivalent to one obtained from the difference of two tilings of a hexagonal region consisting of 6 triangles in the lattice, so  $\pi_2(X_{\mathcal{L}}) = J$ in this case.  Also, since this picture is homologically trivial, $X_{\mathcal{L}}$ is Cockcroft and the tile counting group for the lozenges is isomorphic to $H_2(X_{\mathcal{L}}) = \mathbb{Z}^3$.

\begin{figure}[h]
\centering
\includegraphics[width=2.7in]{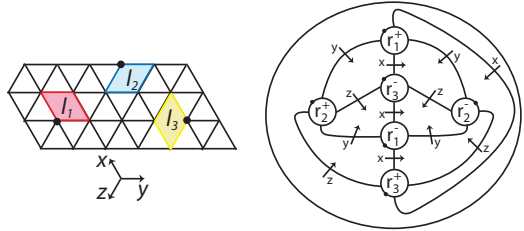}
\caption{Tiling the triangular lattice with lozenges; $\pi_2(X_{\mathcal{L}})$ is generated by a single picture.}\label{fig:lozenge}
\end{figure}

More generally, Bogley shows in \cite{bogley} that the 2-complex associated to $$\mathcal{P} = \langle x, y, z ~|~ [x^a, y], [y^b, z], [z^c,x]\rangle$$ 

\noindent is Cockcroft for positive integers $a, b, c$.  This gives an infinite family of tile sets consisting of 3 elongated lozenges in the triangular lattice for which the tile counting group is $H_2(X_{\mathcal{P}}) = \mathbb{Z}^3$.  We remark that the topological machinery employed in this example is not necessary, and that one may determine the tile counting group directly.

%%%%%%%%%%%%%%%%%%%%  PICTURES AND TILE INVARIANTS

\section{Pictures and Tile Invariants}

Given tile boundary group presentation $$\mathcal{P} = \langle x, y ~|~ r_1, \ldots, r_n \rangle,$$ 

\noindent suppose $\{s_1, s_2, \ldots, s_k\}$ is a set of words in $F(x,y)$ such that each $r_i$ is in the normal subgroup $S$ of $F(x,y)$ generated by the $s_i$.  That is, each $r_i$ is trivial in the group $H = F/S$.  The group $H$ is a quotient of $G_T$ and has presentation $$\mathcal{Q} = \langle x, y ~|~ s_1, \ldots, s_k\rangle.$$
Let $Y$ be the 2-complex modeled on $\mathcal{Q}$.

Since each $r_i$ is trivial in $H$, there is a disk picture $\mathbb{B}_{r_i}$ over $\mathcal{Q}$ having boundary label identically equal to $r_i$, by van Kampen's Lemma.  Given a choice $\{\mathbb{B}_{r_1},\mathbb{B}_{r_2}, \ldots, \mathbb{B}_{r_n}\}$ of such based pictures we may transform a spherical picture over $\mathcal{P}$ to a spherical picture over $\mathcal{Q}$.    In particular, say $\mathbb{P}$ is a spherical picture over $\mathcal{P}$ with disks $\Delta_1, \cdots \Delta_m$, and disk labels $w_i^{\epsilon_i}$ where $w_i \in \{r_1, \cdots r_n\}$ and $\epsilon_i = \pm 1$.  Replace each disk $\Delta_i$ in $\mathbb{P}$ with the disk picture $\mathbb{B}_{w_i}^{\epsilon_i}$.  The resulting spherical picture, call it $\overline{\mathbb{P}}$, gives an element of $\pi_2(Y)$.

If $\pi_2(Y)$ is understood, this map can tell us information about $\pi_2(X_{\mathcal{P}})$, which in turn may give us tile invariants.

This technique of passing to a quotient to arrive at tile invariants was applied by Conway and Lagarias in \cite{conway} and by Pak and Muchnik in \cite{muchnik}.  Their work was couched in terms of winding numbers of paths in the 1-skeleton of the universal covering of the tile complex.  Here we revisit their arguments from the point of view of pictures, and we derive the tile invariants from the fact that the quotient spaces considered are Cockcroft.

\begin{lemma}\label{lemma:dyck} Suppose $Y$ is the 2-complex modeled on $\mathcal{P} = \langle x, y ~|~ x^a, y^b, (xy)^c\rangle$ where $\frac{1}{a}+\frac{1}{b}+\frac{1}{c} \leq 1$.  Then $Y$ is Cockcroft.
\end{lemma}

This result can be found in the literature (see, for instance, \cite[p. 174]{thebook}) but for the reader's convenience we work through the proof at the end of the section.

\begin{figure}[h]
\centering
\centerline{\includegraphics[width=2.8in]{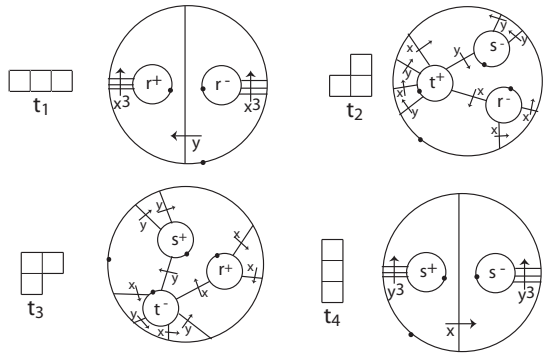}}
\caption{The tiles in $T_3$ determine pictures over $\mathcal{Q} = \langle x, y ~| x^3, y^3, (xy)^3 \rangle$.}\label{fig:T3quotient}
\end{figure}

\begin{theorem}\label{th:T3} {\rm (Conway and Lagarias)} Suppose $\Gamma$ is a region homeomorphic to a disk. For any tiling $\alpha$ of $\Gamma$ by the set of trominoes $T_3$, $b_3(\alpha) - b_2(\alpha) = c(\Gamma)$.
\end{theorem}

\begin{proof} Consider the presentation $\mathcal{Q} = \langle x, y ~| r=x^3, s=y^3, t=(xy)^3 \rangle$ %of the von Dyck group $D(3,3,3)$,
and let $Y$ denote the 2-complex modeled on $\mathcal{Q}$.  By Lemma \ref{lemma:dyck}, $Y$ is Cockcroft.  Moreover, each relator in $\mathcal{P}_3$ determines the trivial element in the group $H$ presented by $\mathcal{Q}$ as demonstrated in Figure \ref{fig:T3quotient}.

Using these based pictures, any two tilings $\alpha$ and $\beta$ of a region $\Gamma$ by $T_3$ determine an element $[\mathbb{B}_\alpha \cdot \mathbb{B}_{\beta}^{-1}]$ in $\pi_2(X_{\mathcal{P}_3})$ whose image $[\overline{\mathbb{B}_\alpha \cdot \mathbb{B}_{\beta}^{-1}}]$ in $\pi_2(Y)$ must be homologically trivial since $Y$ is Cockcroft.   That is, the exponent sum of each relator in $\overline{\mathbb{B}_\alpha \cdot \mathbb{B}_{\beta}^{-1}}$ must be zero.
It follows that for the relator $t$,
$$0 = {\rm exp}_{t}(\overline{\mathbb{B}_\alpha \cdot \mathbb{B}_{\beta}^{-1}}) = \big(b_2(\alpha)-b_3(\alpha)\big)-\big(b_2(\beta)-b_3(\beta)\big).$$  %Since $Y$ is Cockcroft, it follows that ${\rm exp}_{t}(\overline{\mathbb{B}_\alpha \cdot \mathbb{B}_{\beta}^{-1}}) = 0$.
Thus, for any two tilings $\alpha$ and $\beta$ of $\Gamma$ we have $b_2(\alpha)-b_3(\alpha) = b_2(\beta)-b_3(\beta)$.
\end{proof}

We remark that the relators $r_3$ and $r_4$ in the presentation of the tile boundary group for $T_3$ are consequences of the first two (see Figure \ref{fig:T3reduced}). That is, $r_3$ and $r_4$ are trivial in the group presented by $$\mathcal {W}=\langle x, y~|~[y,x^3],(yx)^2y^{-2}x^{-2}\rangle.$$  It would be interesting to know whether the 2-complex $X_{\mathcal{W}}$ associated to this presentation is Cockcroft.  If this is the case, then we have a proof of Theorem \ref{th:T3} that does not involve the quotient group considered in \cite{conway}.  Indeed, if $\alpha$ and $\beta$ are two tilings of a region $\Gamma$ by $T_3$, the spherical picture $\mathbb{B}_\alpha \cdot \mathbb{B}_{\beta}^{-1}$ determines the element $[\overline{\mathbb{B}_\alpha \cdot \mathbb{B}_{\beta}^{-1}}]$ in $\pi_2(X_{\mathcal{W}})$, by using the subpictures of Figure \ref{fig:T3reduced} and one checks that $${\rm exp}_{r_2}([\overline{\mathbb{B}_\alpha \cdot \mathbb{B}_{\beta}^{-1}}]) = (b_2(\alpha)-b_3(\alpha))-(b_2(\beta)-b_3(\beta)).$$

\begin{figure}[h]
\centering
\includegraphics[width=3in]{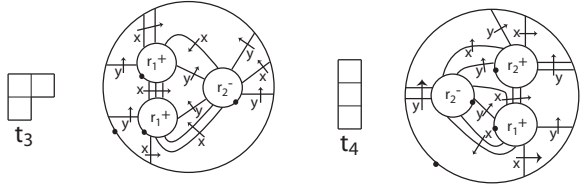}
\caption{Tiles $t_3$ and $t_4$ are consequences of $t_1$ and $t_2$.}\label{fig:T3reduced}
\end{figure}

\begin{theorem}\label{th:T4} {\rm (Pak and Muchnik)} Suppose $\Gamma$ is a region homeomorphic to a disk.  For any tiling $\alpha$ of $\Gamma$ by the set of tetrominoes $T_4$, $b_2(\alpha)+b_4(\alpha)-b_5(\alpha)-b_7(\alpha) = c(\Gamma)$.
\end{theorem}

\newpage
\begin{proof} Consider the presentation $$\mathcal{Q}=\langle x, y ~|~ r=x^4, s=y^4, t=(xy)^2\rangle$$
%of the von Dyck group $D(4,4,2)$,
and let $Y$ be the 2-complex modeled on $\mathcal{Q}$.  By Lemma \ref{lemma:dyck}, $Y$ is Cockcroft.  Moreover, each relator in ${\mathcal{P}}_4$ determines the trivial element in the group $H$ presented by $\mathcal{Q}$ as demonstrated in Figure \ref{fig:T4quotient}.  The result then follows by the same argument as in the proof of Theorem \ref{th:T3}, making use of the relator $r$.  In this case, observe that 
\begin{align*}
{\rm exp}_{r}(\overline{\mathbb{B}_\alpha \cdot \mathbb{B}_{\beta}^{-1}}) = &\big(-b_2(\alpha)-b_4(\alpha) + b_5(\alpha)+b_7(\alpha)\big)\\
&-\big(-b_2(\beta)-b_4(\beta) + b_5(\beta)+b_7(\beta) \big).\\
  \end{align*}  
  \vskip-.2in \noindent Since this exponent sum must be 0, it follows that for any two tilings $\alpha$ and $\beta$ of $\Gamma$ we have $b_2(\alpha)+b_4(\alpha)-b_5(\alpha)-b_7(\alpha) = b_2(\beta)+b_4(\beta)-b_5(\beta)-b_7(\beta)$.\end{proof}
\begin{figure}[h]
\centering
\centerline{\includegraphics[width=3.5in]{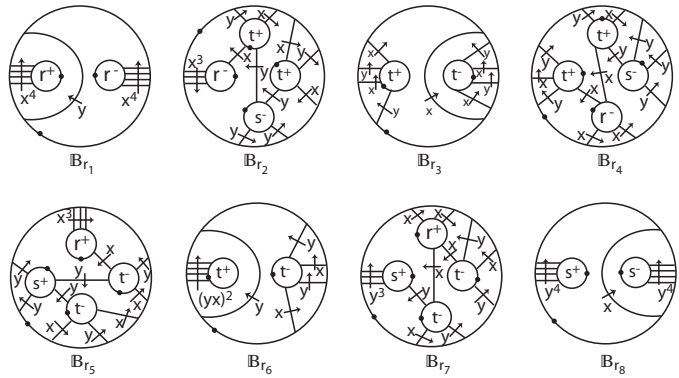}}
\caption{Tile $t_i \in T_4$ determines picture $B_{r_i}$ over $\mathcal{Q}=\langle x, y ~|~ x^4, y^4, (xy)^2\rangle$.}\label{fig:T4quotient}
\end{figure}

We may perhaps use Lemma \ref{lemma:dyck} to reverse engineer other interesting tile sets.  Consider the group presented by $$\mathcal{Q} = \langle x, y ~|~ x^3, y^6, (xy)^2 \rangle.$$  One can check that $\mathcal{Q}$ presents a quotient of the tile boundary group associated with the tiles in Figure \ref{fig:6ominoes}.  Using techniques from this section it is straight forward to check that for any tiling $\alpha$ of a region $\Gamma$ by this tile set, $b_2(\alpha) - b_3(\alpha)$ is constant.

\begin{figure}[h]
\centering
\centerline{\includegraphics[width=2.1in]{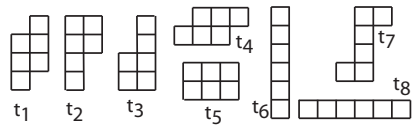}}
\caption{A tile set consisting of 6-ominoes}\label{fig:6ominoes}
\end{figure}

{\em Proof of Lemma \ref{lemma:dyck}}.  We first argue that any connected spherical picture over $\mathcal{P}$ must have a local configuration of one of four types in Figure \ref{fig:dipoles}.  We call such a configuration a dipole.

\begin{figure}[h]
\centering
\centerline{\includegraphics[width=1.3in]{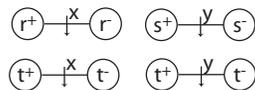}}
\caption{Local configurations in a picture over $\mathcal{Q} = \langle x, y ~| r=x^a, s=y^b, t=(xy)^c \rangle$}\label{fig:dipoles}
\end{figure}

Suppose we have a connected spherical picture $\mathbb{P}$ over the presentation $\mathcal{P}$, and we shrink each disk to a vertex.  What remains is a labeled graph on the sphere.  Each edge gets labeled with $x$ or $y$; each vertex by $r=x^a$, $s=y^b$, or $t=(xy)^c$.   One can define an angle $\alpha$ at each corner of the graph and subsequently a curvature function $\gamma$ on the vertices and faces of this graph as follows: $\gamma(v) = 2\pi - \sum \alpha_i$ and $\gamma(f) = 2\pi - \sum(\pi - \alpha_i)$ where in each case the sum runs over all angles about the vertex or over all interior angles of the face, respectively.  Let $V$, $E$, and $F$ denote the number of vertices, edges, and faces in the resulting graph.  Since the graph will have twice as many angles as edges one can check that $$\sum_{{\rm all~}v}\gamma(v) + \sum_{{\rm all~}f}\gamma(f) = 2\pi(V-E+F)= 2\pi \chi(\mathbb{S}^2) = 4\pi.$$  It follows that some face or vertex of the graph must have positive curvature.

Assign an angle to each corner at each vertex according to this rule: If the corner consists of two $x$ edges, assign the angle $2\pi/a$.  If the corner consists of two $y$ edges, assign the angle $2\pi/b$, and if the corner consists of one $x$ edge and one $y$ edge, assign the angle $\pi/c$.  With this angle assignment, the curvature at each vertex in the graph obtained from $\mathbb{P}$ is 0.

Now, we assume $\mathbb{P}$ has no dipole and we argue that the curvature of any face will be less than or equal to 0.  To see this, note that any vertex labeled $r$ or $s$ can only have $t$ vertices as neighbors, and vice versa.  It follows that any face in the graph will be formed by $4k$ vertices for some $k \geq 1$ having vertex label sequence $(rtst)(rtst)\cdots(rtst)$.  The curvature of such a face $F$ can be summed by grouping the corners by vertex label.  This face will have $k$ vertices labeled by $r$, $k$ labeled by $s$ and $2k$ labeled by $t$, so

\begin{align*}
\gamma(f)&= 2\pi - k(\pi-2\pi/a) - k(\pi-2\pi/b)-2k(\pi-\pi/c)\\
&=2\pi(1-2k+k(1/a+1/b+1/c))\\
&\leq 2\pi(1-k) ~~~{\rm since}~ 1/a+1/b+1/c\leq 1\\
&\leq 0   ~~~{\rm since}~ k \geq 1.
\end{align*}

But this contradicts the fact that some disk or region must have positive curvature, so $\mathbb{P}$ must have a dipole.  By a series of bridge moves the disks in the dipole can be split off from $\mathbb{P}$.  In this way $\mathbb{P}$ is equivalent to a finite sum of spherical pictures, each one consisting of two disks with the same relator label but oppositely oriented.  Thus, the image of $[\mathbb{P}]$ under the Hurewicz map is trivial, and $Y$ is Cockcroft.  \qed

In fact, this argument tells us a bit more, namely that $\pi_2(Y)$ is generated as a left $\mathbb{Z}\pi_1(Y)$-module by the spherical pictures in Figure \ref{fig:pi2dyck}.
\begin{figure}[h]
\centering
\centerline{\includegraphics[width=2in]{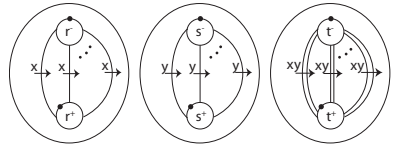}}
\caption{Generators of $\pi_2(Y_{\mathcal{Q}})$, where $\mathcal{Q} = \langle x, y ~| r=x^a, s=y^b, t=(xy)^c \rangle$ and $1/a+1/b+1/c\leq 1$.}\label{fig:pi2dyck}
\end{figure}

%%%%%%%%%%%%%%%%%%%%% A TILING QUESTION

\section{A Tiling Question}

In \cite[Theorem 7.1]{pak} Pak determined which $a \times b$ rectangles could be tiled by the subset $S = \{t_2, t_3, t_5, t_8\}$ of $T_4$ given in Figure \ref{fig:S}.  In the spirit of tiling questions involving mutilated checkerboards, we determine which modified rectangles $M(a,b)$ (obtained by removing the upper left and lower right squares of an $a$ by $b$ rectangle) can be tiled by $S$.

\begin{figure}[h]
\centering
\centerline{\includegraphics{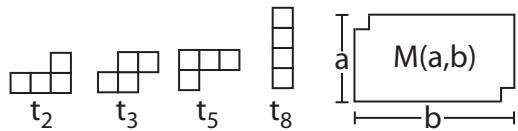}}
\caption{The tile set $S$}\label{fig:S}
\end{figure}

The key to Pak's argument, and to ours, is that if a region $\Gamma$ can be tiled by $T_4$, then the number of tiles in any tiling of $\Gamma$ that come from $S$ must be constant, modulo 2.  This invariant, called the {\em height invariant} in \cite{pak}, can be stated as follows:  if $\Gamma$ can be tiled by $T_4$ and $\alpha$ is a tiling of $\Gamma$ by $T_4$ then
\begin{equation*}
b_2(\alpha) + b_3(\alpha) + b_5(\alpha) + b_8(\alpha) \equiv h(\Gamma)  ~({\rm mod~} 2)
\end{equation*}
where $h(\Gamma)$ is a constant depending only on the region $\Gamma$.  The height invariant can be derived by combining the invariant of Theorem \ref{th:T4} with an invariant that comes from a coloring argument. We work through this derivation following the proof of Lemma \ref{lemma:th4}.  Now we wish to emphasize its usefulness.

In the context of our tiling question and the set $S$, we say that $T_4$ {\em oddly tiles} a region $\Gamma$ if $T_4$ tiles $\Gamma$ in such a way that an odd number of the tiles used come from the complementary set $T_4 \setminus S$.  Figure \ref{fig:mutex} shows three modified rectangles tiled by $S$, and two regions oddly tiled by $T_4$.

The following ``odd tiling'' lemma gives us a constructive tool for proving the non-tileability of a region $\Gamma$ by the set $S$.

\smallskip

\begin{lemma} \label{lemma:oddtile}  {\em If $T_4$ oddly tiles $\Gamma$ then $S$ does not tile $\Gamma$.}
\end{lemma}

\begin{proof} In any tiling of a region $\Gamma$ by $T_4$ the total number of tiles used is constant.  If $T_4$ oddly tiles $\Gamma$, then $T_4$ tiles $\Gamma$ with $m + k$ tiles where $m \geq 0$ is the number of tiles from $S$ and $k \geq 1$ is the number of tiles from $T_4\setminus S$ with $k$ odd.  Then $h(\Gamma) \equiv m$ (mod 2).  If $S$ itself tiles $\Gamma$ it would have to do so with $m+k$ tiles and we would have $h(\Gamma) \equiv m+k$ (mod 2) as well.  But since $k$ is odd, this would be a contradiction.
\end{proof}

\begin{figure}[h]
\centering
\includegraphics[width=3.3in]{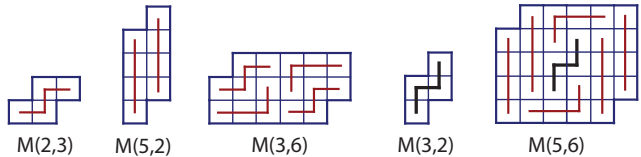}
%{mutexample}
\caption{Three modified rectangles tiled by $S$, and two oddly tiled by $T_4$.}
\label{fig:mutex}
\end{figure}

\begin{theorem}\label{th:4} Let $a, b > 1$.  The set $S$ tiles $M(a,b)$ if and only if

\noindent {\rm 1.} $a \equiv 2$ {\rm (mod 4)} and $b$ is odd; or
\par\noindent {\rm 2.} $a$ is odd and $ab \equiv 2$ {\rm (mod 8)}.
\end{theorem}

\begin{proof}

We begin by showing that tilings exist in these two cases.
First, suppose $a \equiv 2$ (mod 4) and $b$ is odd. Observe that $S$ tiles $M(2,3)$, as indicated in Figure \ref{fig:mutex}.
If $S$ tiles $M(2,b)$ then $S$ tiles $M(2,b+2)$ by tacking on a copy of the tile $t_3$, as suggested in Figure \ref{fig:M2b}.  It follows inductively that $S$ tiles $M(2,b)$ for any odd $b \geq 3$.  Furthermore, if $S$ tiles $M(a,b)$ then $S$ tiles $M(a+4,b)$ by Lemma \ref{lemma:th4}.  Thus, $S$ tiles $M(a,b)$ if $a = 2$ (mod 4) and $b$ is odd.

\begin{figure}[h]
\centering
\includegraphics[width=1.2in]{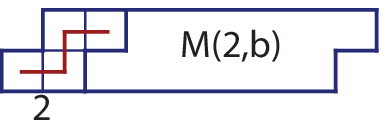}
\caption{If $S$ tiles $M(2,b)$ then $S$ tiles $M(2,b+2)$.}\label{fig:M2b}
\end{figure}

Next, suppose $a$ is odd and $ab \equiv 2$ (mod 8). Then either
$a \equiv$ $1 ~{\rm(mod ~4)}$ and $b \equiv 2$ (mod 8); or
$a \equiv 3$ (mod 4) and $b \equiv 6$ (mod 8).
Observe that $S$ tiles $M(5,2)$ and $M(3,6)$, as in Figure \ref{fig:mutex}.   Using Lemma \ref{lemma:th4} it follows by induction that $S$ tiles $M(a,b)$ in both of these subcases.  That is, $S$ tiles $M(a,b)$ if $a$ is odd and $ab \equiv 2$ (mod 8).

\smallskip
Now we prove that $S$ does not tile $M(a,b)$ for any other values
of $a$ and $b$. If $S$ tiles $M(a,b)$ then $ab - 2$ is divisible
by 4. That is, $ab \equiv 2$ (mod 4)  so either $a \equiv 2$ (mod 4) and $b \geq 3$ is odd, or $a \geq 3$ is odd and $b \equiv 2$ (mod 4).  By considering $a$ modulo 4 and $b$ modulo 8, the $a$ is odd case can be subdivided into these four subcases:

\noindent (i) $a \equiv 1$ (mod 4) and $b \equiv 2$ (mod 8)

\noindent  (ii) $a \equiv 1$ (mod 4) and $b \equiv 6$ (mod 8)

\noindent  (iii) $a \equiv 3$ (mod 4) and $b \equiv 2$ (mod 8)

\noindent (iv) $a \equiv 3$ (mod 4) and $b \equiv 6$ (mod 8).

\noindent We use induction to rule out subcases ii and iii.  For the base cases, observe that $T_4$ oddly tiles $M(5,6)$ and $M(3,2)$, as shown in Figure \ref{fig:mutex}.  Then Lemma \ref{lemma:th4} provides the inductive step: if $T_4$ oddly tiles a region $M(a,b)$ then it oddly tiles $M(a+4,b)$ and $M(a,b+8)$.  Thus, $T_4$ oddly tiles $M(a,b)$ in both subcases ii and iii, so $S$ itself does not tile $M(a,b)$ in these cases.

Thus, if $S$ tiles $M(a,b)$ then either $a \equiv 2$ (mod 4) and $b$ is odd, or $a$ is odd and $ab \equiv 2$ (mod 8).
\end{proof}

\begin{lemma}\label{lemma:th4} Let $a,b > 1$.  If $T_4$ oddly tiles $M(a,b)$ then $T_4$ oddly tiles $M(a+4,b)$ and $M(a,b+8)$.  Moreover, if $S$ tiles $M(a,b)$, then $S$ tiles $M(a+4,b)$ and $M(a,b+8)$.
\end{lemma}

\begin{proof}

The region $M(a+4,b)$ may be partitioned as in Figure \ref{fig:lemmath4} (a) into two regions: $M(a,b)$ and a region above this that may be tiled by $S$ using $b$ copies of the vertical tile $t_8$.  In this way a tiling (resp. odd tiling) of $M(a,b)$ by $S$ (resp. $T_4$) may be extended to a tiling (resp. odd tiling) of $M(a+4,b)$ by $S$ (resp. $T_4$).

\begin{figure}[h]
\centering
\includegraphics[width=2in]{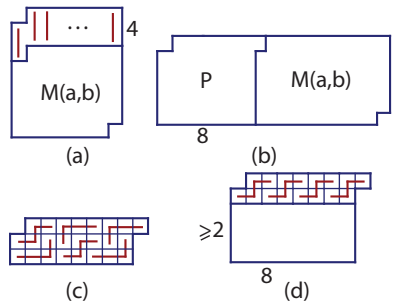}
\caption{(a) Going from $M(a,b)$ to $M(a+4,b)$; (b) Going from $M(a,b)$ to $M(a,b+8)$; (c) tiling $P$ when $a = 3$; (d) tiling $P$ when $a \geq 4$.}\label{fig:lemmath4}
\end{figure}
The region $M(a,b+8)$ may be partitioned as in Figure \ref{fig:lemmath4} (b).  Region $P$ in this partition may be tiled by $S$ for all $a \geq 2$.  Indeed, if $a =2$, then $P$ is $M(2,9)$ and may be tiled with 4 copies of the tile $t_3$.  If $a = 3$ then $P$ may be tiled by $S$ as in Figure \ref{fig:lemmath4} (c), and if $a \geq 4$ then $P$ may be tiled by $S$ according to Figure \ref{fig:lemmath4} (d), noting that $S$ tiles the rectangle in the figure since its width is 8 and its height is at least 2.  (To see this, find a tiling of the $2 \times 8$ and $3 \times 8$ rectangle by $S$.)  In this way a tiling (respectively, odd tiling) of $M(a,b)$ by $S$ (resp. $T_4$) may be extended to a tiling (respectively, odd tiling) of $M(a,b+8)$ by $S$ (respectively, $T_4$).
\end{proof}

We now derive the height invariant by first considering a coloring argument.  Following \cite[Section 2]{moore}, ``color" the squares of the lattice in horizontal swaths using integers 0, 1, 2, 3 as indicated in Figure \ref{fig:4coloring}.  Note that for each tile in $T_4$, the sum of its square colors, wherever the tile is placed in the lattice, is constant modulo 4.  That is, this coloring produces a well-defined function $\phi:\mathcal{R} \to \mathbb{Z}_4$, where $\phi(\Gamma)$ equals the sum of the colors in $\Gamma$, modulo 4.  Moreover, we have $\phi(t_1)=0$, $\phi(t_2)= 1$, $\phi(t_3)=2$, $\phi(t_4)=3, \phi(t_5)=3$, $\phi(t_6)= 0$, $\phi(t_7)=1$, $\phi(t_8)=2$, which yields a tile invariant that Pak calls the shade invariant in \cite{pak}. If $\alpha$ denotes a tiling of $\Gamma$ by $T_4$ then $$b_2(\alpha) + 2b_3(\alpha) +3b_4(\alpha) + 3b_5(\alpha)+b_7(\alpha)+2b_8(\alpha)\equiv \phi(\Gamma)~~({\rm mod~} 4).$$

Combine this shade invariant with the tile invariant in Theorem \ref{th:T4} and view the result modulo 2 to obtain the height invariant.

\begin{figure}[h]
\centering
\centerline{\includegraphics[width=.9
in]{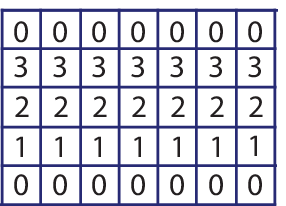}}
\caption{A coloring of the integer lattice}\label{fig:4coloring}
\end{figure}

We close this discussion by showing that Thoerem \ref{th:4} cannot have been proved using a coloring argument.  A {\em signed tiling} of a region $\Gamma$ by a set $T$ of tiles consists of a placement of tiles from $T$, each given a weight of +1 or -1, in such a way that the sum of the weights of the tiles is +1 for each square inside the region $\Gamma$, and 0 for each square outside the region.  There is a simple connection between coloring arguments and signed tilings, as developed in \cite[Theorem 5.2]{conway} and \cite[Theorem 8.1]{pak}.  In particular, if a coloring argument can be used to show a region is not tileable by $T$ then there is no signed tiling of the region by $T$.  This means that (a) if you know a region $\Gamma$ cannot be tiled by $T$, and (b) there exists a signed tiling of $\Gamma$ by $T$, then the untileability of $\Gamma$ by $T$ cannot have been proved using a coloring argument.

\begin{lemma} \label{lemma:th5} If $T_4$ oddly tiles a simply connected region $\Gamma$, then there exists a signed tiling of $\Gamma$ by $S$.
\end{lemma}

\begin{proof} Figure \ref{fig:nocolor} shows a signed tiling of each tile in $T_4\setminus S$ by the set $S$.  Then, if $T_4$ oddly tiles a region $\Gamma$, assign weight +1 to each tile in the odd tiling that comes from $S$, and replace each of the tiles from $T_4\setminus S$ with a signed tiling of that tile by the set $S$.  This produces a signed tiling of $\Gamma$ by $S$.
\end{proof}

\begin{figure}[h]
\centering
\centerline{\includegraphics[width=2.5in]{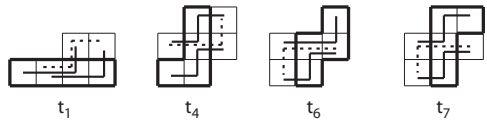}}
\caption{Signed tilings of each tile in $T_4\setminus S$ by the set $S$.}\label{fig:nocolor}
\end{figure}

In the proof of Theorem \ref{th:4} we found that $T_4$ oddly tiled $M(a,b)$ in various cases.  By Lemma \ref{lemma:th5} it follows that there are signed tilings of $M(a,b)$ by $S$ in these cases, and the following result is now immediate.

\begin{theorem} Theorem \ref{th:4} cannot be proved with a coloring argument.
\end{theorem}

\medskip
{\bf Acknowledgments.}
The author would like to thank Bill Bogley for our discussions as this paper took shape, and the reviewers for their helpful comments and suggestions.

%%%%%%%%%%%%%%%%%%%%%%%% BIBLIOGRAPHY

\end{document}